\documentclass[oneside,a4paper,11pt]{amsart}
\usepackage{amssymb,amsmath,t1enc,amsthm,geometry,enumitem}
\usepackage[colorlinks]{hyperref}
\usepackage[utf8]{inputenc} 
\usepackage[english]{babel} 
\usepackage{color}
\usepackage{enumitem}
\usepackage{amscd} 
\usepackage{xy}
\usepackage{xypic}
\usepackage{graphicx}

\hypersetup{linkcolor = blue,
            urlcolor  = blue,
            citecolor = green,
            anchorcolor = green}

\newcommand{\halmaz}[1]{\left\{\,#1\,\right\}}

\newcommand*{\C}{\mathcal C}
\newcommand*{\TM}{\mathcal TM}
\newcommand*{\T}{\mathcal T}

\newcommand{\R}{\mathbb R}
\newcommand*{\F}{\mathcal{F}}

\newcommand{\norm}[1]{\ensuremath{\left\lvert #1 \right\rvert}}
\newcommand{\innerp}[1]{\ensuremath{\left\langle #1 \right\rangle}}
\newcommand*{\X}[1]{{\mathfrak X}(#1)}
\newcommand*{\h}{\mathfrak h}

\def\nr{\nabla^{^R}}
\def\P{\mathcal P}

\numberwithin{equation}{section} %% Comment out for sequentially-numbered

\theoremstyle{plain}
\newtheorem{theorem}{Theorem}[section]
\newtheorem{proposition}[theorem]{Proposition}
\newtheorem{lemma}[theorem]{Lemma}
\newtheorem{corollary}[theorem]{Corollary}
\newtheorem{property}[theorem]{Property}

\theoremstyle{definition}
\newtheorem{definition}[theorem]{Definition}
 
\newtheorem{remark}[theorem]{Remark}

\newtheorem{example}[theorem]{Example}

\begin{document}

\title[Natural parallelism associated to navigation data]{Natural parallel
  translation and connection associated to navigation data}

\author{A.~Mezrag}
\author{Z.~Muzsnay}
\author{Cs.~Vincze}

\address{Asma Mezrag, Institute of Mathematics, University of Debrecen,
  Debrecen, Hungary} \email{asma1998mezrag@gmail.com}

\address{Zolt\'an Muzsnay, Institute of Mathematics, University of
  Debrecen, Debrecen, Hungary} 
\urladdr{https://math.unideb.hu/en/dr-zoltan-muzsnay}
\email{muzsnay@science.unideb.hu}

\address{Csaba Vincze, Institute of Mathematics, University of Debrecen,
  Debrecen, Hungary}
\urladdr{https://math.unideb.hu/dr-vincze-csaba}
\email{csvincze@science.unideb.hu}

\date{\today}



\begin{abstract}
  In this paper, we consider the geometric setting of navigation data and
  introduce a natural parallel translation using the Riemannian
  parallelism. The geometry obtained in this way has some nice and natural
  features: the natural parallel translation is homogeneous (but in general
  nonlinear), preserves the Randers type Finslerian norm constituted by the
  navigation data, and the holonomy group is finite-dimensional.
\end{abstract}

\maketitle

\section{Introduction}
\label{sec:1}

The parallel translation is one of the most important geometric concept
appearing in different geometric settings, such as Riemannian geometry,
Finsler geometry, and relativity theory. It can be introduced through
different objects, such as covariant derivatives
\cite{Kobayashi_Nomizu_1996,Spivak_1979}, or different types of connections
\cite{Ehresmann_1995,Grifone_1972}. We also have an associated algebraic
structure, called holonomy group, by considering parallel translations
along loops. The holonomy group can provide important information about the
geometry of the manifold, see for example the de Rham decomposition theorem
\cite{de_rham_1952,Eschenburg_Heintze_1998}.  In some cases, the holonomy
group is a finite-dimensional group, but in the case of homogeneous
(nonlinear) parallel translations there are examples, where the holonomy
group is infinite-dimensional \cite{ Mezrag_Muzsnay_2024,
  Muzsnay_Nagy_max_2015}. It turns out that, in the case of Finsler
manifolds, the holonomy group of the the natural parallel translation is,
in general, an infinite-dimensional group
\cite{Hubicska_Matveev_Muzsnay_2021}. This result motivates the
investigation of homogeneous parallel translations where the holonomy group
is finite-dimensional. As we will see, such a geometric setting can be
associated with navigation data on a manifold.

The navigation data consists of a pair $(h,W)$, where $h$
is a Riemannian metric, and $W$ is a smooth vector field on a manifold $M$.  The Zermelo's 
navigation problem is to find the paths of shortest travel time in a
Riemannian manifold $(M, h)$, under the influence of a wind or a current
represented by a vector field $W$.  As D.~Bao, C.~Robles, and Z.~Shen
proved in \cite{Bao_Robles_Shen_2004} that the Zermelo's navigation problem is
equivalent to considering geodesics of Randers-type Finsler metrics.  The
construction of the metric structure associated to the navigation data is easy
to understand (the sets of unit vectors, called indicatrices are blown away
by the wind $W$), however, the affine structure (the parallel translation)
is not so easy or natural to understand \cite{robles_2007}. Moreover, the
holonomy group can be very large even in cases when the metric structure
is relatively simple \cite{Hubicska_Muzsnay_2020a}.

For this reason, we consider the geometric setting of navigation data and
introduce a natural parallel translation using the Riemannian
parallelism. The geometry obtained in this way has some nice and
natural features: the natural parallel translation is homogeneous (but in general nonlinear), preserves
the Randers type Finslerian norm constituted by the navigation data, and the holonomy group is finite-dimensional.

\section{Preliminaries}
\label{sec:2}

Throughout this article, $M$ denotes a connected differentiable manifold of class $C^\infty$, $\X{M}$
is the vector space of smooth vector fields on $M$.  The first and the
second tangent bundles of $M$ are denoted by $(TM,\pi ,M)$ and
$(TTM,\tau ,TM)$, respectively.  Local coordinates $(x^i)$ on $M$ induce
local coordinates $(x^i, y^i)$ on $TM$.  The vector 1-form $J$ on $ TM$,
defined locally as $ J = \frac{\partial}{\partial y^i} \otimes dx^i$, is
known as the natural almost-tangent structure of $TM$. Additionally, the
vertical vector field $\C = y^i \frac{\partial}{\partial y^i}$ on $TM$ is
called the Liouville vector field.

\subsection{Connections}

In this section we recall the differential algebraic presentation of
the connection theory introduced in \cite{Grifone_1972}, which we 
constantly use in the following.

\begin{definition}
  \label{def:connection}
  A \emph{connection} on $M$ is a tensor field of type (1-1) $\Gamma $ on
  $TM$ ( i.e. $\Gamma \in \Psi ^1(TM))$ such that $J \Gamma = J$, and
  $\Gamma J = - J$.  The connection is called \emph{homogeneous} if
  $[\mathcal C,\Gamma] = 0$, it is ${\mathcal C}^\infty$ on
  $TM \setminus \{0\}$ and ${\mathcal C}^0$ on $TM$. In particular, if
  $\Gamma $ is ${\mathcal C}^1$ on the tangent manifold $TM$ (including the
  $0$ section), then it is called \emph{linear}.
\end{definition}
If $\Gamma $ is a connection, then $\Gamma ^2 =I$ and the eigenspace
corresponding to the eigenvalue $-1$ is the vertical space. Then, at any nonzero $z\in TM$, we have the splitting
\begin{equation}
  \label{eq:T_h_v}
  T_zTM = \mathcal H_z \oplus \mathcal V_z,
\end{equation}
where the so-called \emph{horizontal space} $H_z$ is the eigenspace
corresponding to the eigenvalue $+1$. The matrix of $\Gamma$ in the basis
\begin{math}
  \bigl\{\frac {\partial}{\partial x^\alpha },\frac {\partial}{\partial
    y^\alpha } \bigl\}
\end{math}
is
\begin{equation}
  \label{eq:Gamma}
  \Gamma = 
  \left (
    \begin{array}{cl}
      \ \delta _\alpha ^\beta &\ \ 0
      \\
      -2\Gamma _\alpha ^\beta  &-\delta _\alpha ^\beta 
    \end{array}
  \right) ,
\end{equation}
where $\Gamma_\alpha^\beta = \Gamma_\alpha^\beta(x,y)$ are functions called
\emph{coefficients of the connection}. If the connection is homogeneous
(resp. linear), then the coefficients $\Gamma_\alpha^\beta(x,y)$ are homogeneous of degree one (resp.~linear) in $y$. The \emph{horizontal} and \emph{vertical} projectors are denoted by 
\begin{equation}
  \label{eq:h_v}
  \mathfrak h : = \tfrac{1}{2}(Id + \Gamma ), \qquad
  \nu  : = \tfrac{1}{2}(Id - \Gamma ),
\end{equation}
respectively. In terms of local coordinates we have:
\begin{displaymath}
  \h \Bigl(\frac{\partial}{\partial x^\alpha }\Bigl)  =
  \displaystyle \frac{\partial}{\partial x^\alpha } - \Gamma _\alpha
  ^\beta  \frac{\partial}{\partial y^\beta },
  \qquad
  \h \Bigl( \frac{\partial}{\partial y^\alpha }\Bigl)  =0 .
\end{displaymath}

\begin{definition}
  \label{def:cov_der}
  Given two manifolds $N$ and $M$, let $w \in {\frak X}(N)$ and $z\in \frak{X}_N(M)$, that is $z\colon N\to TM$ be smooth vector  fields. Using the natural isomorphism 
  \begin{math}
    \xi_z :T_z^v TM\to T_{\pi(z)}M,
  \end{math}
 the \emph{covariant derivative} of $z$ with
  respect to $w$ is defined by:
  \begin{equation}
    D_{w(x)} z = \xi_{z(x)}( \nu {\scriptstyle {\circ}}
    z_*{\scriptstyle {\circ}} w).
  \end{equation} 
\end{definition}
We have the following diagram:
\begin{equation}
  \label{diag:cov}
  \diagram %%
  & TTM \rto^{\nu} \dto & T^vTM \dlto^{\scriptstyle \xi_z}
  \\
  TN \urto^{\scriptstyle z_*} \rto \dto^{\pi} & TM \dto^{\pi}
  \\
  N \uto<1ex>^{\scriptstyle w} \rto_{\pi \circ z} \urto^z
  \urto<1ex>_{\scriptstyle D_wz} & M \enddiagram
\end{equation}

\medskip

\noindent
In particular, for $N=M$ and $w, z \in {\frak X}(M)$, we have 
\begin{equation}
  D_wz= w^\alpha
  \left(\frac{\partial z^\beta }{\partial x^\alpha } + \Gamma
    ^\beta _\alpha \bigl(x, z(x) \bigr)\right)\frac{\partial}{\partial x^\beta } .
\end{equation}
If $N = [a,b]$ is an interval of $\Bbb R$, $w = \frac{d}{dt}$ and
$z : [a,b] \to TM$, $z(t) = \bigl(x(t),y(t)\bigr) $ is a vector field along
a curve $\gamma\colon [a,b] \to M$ (that is $\gamma = \pi\circ z$), we
arrive at:
\begin{equation}\label{parallelcurve}
  D_{\frac{d}{dt}}z =
  \left
    (\frac{dy^\beta }{dt} + \Gamma ^\beta _\alpha (x
    (t),y(t))\frac{dx ^\alpha }{dt}  
  \right)
  \frac {\partial}{\partial x^\beta }.
\end{equation} 

\begin{definition}
  \label{def:parallel}
  A vector field $z$ along a curve $\gamma$ is called
  \emph{parallel} if $D_{\frac{d}{dt}}z = 0$.  An
  \emph{autoparallel curve} is a curve $\gamma\colon [a,b] \to M$ such that
  $D_{\frac{d}{dt}}\gamma ' = 0.$
\end{definition}
\begin{remark}
  \label{rem:parallel}
  From Definitions \ref{def:cov_der} and \ref{def:parallel} we get that a vector field $z$ is
  parallel along a curve $\gamma$ if and only if $\nu \circ z' = 0$, that is
  $z'$ is horizontal, and a curve $\gamma $ is autoparallel if and only
  if
  \begin{math}
    \nu \circ \gamma '' = 0,
  \end{math}
  that is $\gamma ''$ is horizontal.
\end{remark}

\subsection{Spray and autoparallel curves}

A vector field $S\in \mathfrak{X}(TM)$ is called a spray if $JS = \C$ and
$[\C, S] = S$. Locally, a spray can be expressed as follows
\begin{equation}
  \label{eq:spray}
  S = y^i \frac{\partial}{\partial x^i} - 2G^i\frac{\partial}{\partial y^i},
\end{equation}
where the \emph{spray coefficients} $G^i=G^i(x,y)$ are $2$-homogeneous
functions in the fiber coordinates.

\begin{definition}
  \label{def:assoc_spray}
  Let $\Gamma $ be a homogeneous connection, and consider its horizontal projector $\h$.  The
  spray $S$ associated to the connection is defined by $S = \h {\tilde S}$,
  where ${\tilde S}$ is an arbitrary spray.
\end{definition}
Locally, the coefficients of the spray associated to the connection
\eqref{eq:Gamma} are
\begin{equation}
  \label{eq:G_Gamma}
  G^\alpha= \tfrac{1}{2}y^\beta \Gamma^\alpha_\beta,
\end{equation}
where the $\Gamma^\alpha_\beta (x,y)$ are the coefficients of the
connection.  Conversely, we have a homogeneous connection $\Gamma :=[J,S]$ associated to a spray $S$. If the spray has a local expression \eqref{eq:spray},
then the coefficients of the connection are:
\begin{displaymath}
  \Gamma^\alpha_\beta = \frac{\partial G^\alpha }{\partial y^\beta}.
\end{displaymath}

The notion of sprays allows us to speak about a system of second order
differential equations in a coordinate free way as follows: a
parametrized curve $c$ is called a \emph{path} of the spray $S$ if its
velocity field $\dot{c}$ is an integral curve of $S$, that is:
\begin{equation}
  \label{eq:S_c}
  S \circ {\dot{c}} = \ddot{c}.
\end{equation}
Using local coordinates, $c$ is a path of the spray \eqref{eq:spray} if
and only if it satisfies the system of second order differential equations:
\begin{equation}
  \label{eq:path}
  \frac{d^2x^\alpha }{dt^2} + 2G^\alpha \Bigl(x, \frac{dx}{dt} \Bigl)=0.
\end{equation}
It is easy to verify that the paths of the spray associated to the
connection $\Gamma$ are the autoparallel curves of $\Gamma$.

\medskip

\subsection{Navigation data and the associated Finsler structure}

\begin{definition}
  A \emph{Finsler manifold} of dimension $n$ is a pair $(M,F)$, where $M$
  is a differentiable manifold of dimension $n$ and $F: TM \to \R$
  satisfies: \vspace{-0pt}
  \begin{enumerate}[label=(\alph*), leftmargin=25pt, itemsep=-0pt]
  \item $F$ is smooth and strictly positive on
    $\TM:= TM \setminus \halmaz{0}$,
  \item $F$ is positively homogeneous of degree $1$ in the directional
    argument $y$,
  \item the metric tensor
    $g_{ij}= \frac{1}{2}\frac{\partial^2 E}{\partial y^i \partial y^j}$ has
    rank $n$ on $\TM$, where $E:=\frac{1}{2}F^2$ is the energy
    function. \vspace{-5pt}
 \end{enumerate}
\end{definition}
\noindent
The pair $(h,W)$ is called a navigation data on $M$, if $h$ is a Riemannian
metric and $W$ is a smooth vector field with Riemannian length
$\norm{W}^2 < 1$.  We denote by $\innerp{X, Y}=h_x(X,Y)$ and
$\norm{X}^2 = h_x(X,X)$ the inner product and the norm defined by the
Riemannian metric $h$.

\medskip

The Randers-type Finslerian norm function associated to the navigation data
$(h, W)$ is defined as
\begin{equation}
  \label{eq:randers_norm}
  \F(x,y) = \frac{\sqrt{ \innerp{y, W}^2 + \lambda \norm{y}^2 }}{\lambda} -
  \frac{\innerp{y, W}} {\lambda}
\end{equation}
where $x\in M$, $y\in T_xM$, and $\lambda$ is the smooth function
\begin{equation}
  \label{eq:randers_lambda}
  \lambda = 1-\norm{W}^2.
\end{equation}

\begin{remark}
  Let $\F$ be the Randers norm associated to the navigation data $(h,W)$.
  At any point $p\in M$ a vector $V^\circ_p$ has a unit Finsler norm if and
  only if $V^\circ_p-W_p$ has a unit Riemannian norm, that is
  \begin{equation}
    \label{eq:unit_norm}
    \F(V^\circ_p) = 1 \qquad \Leftrightarrow \qquad
    \norm{V^\circ_p-W_p} = 1. 
  \end{equation}
\end{remark}

\section{Natural palallel translation associated to navigation data}
\label{sec:3}

\begin{definition}
  \label{def:nat_parallel}
  Let $c$ be a curve from $p$ to $q$, and suppose that $V^\circ_p$ a unit
  vector with respect to the Randers norm function \eqref{eq:randers_norm}
  associated to the navigation data $(h, W)$. We define the \emph{natural
    parallel translation} associated to the navigation data $(h, W)$ of
  $V^\circ_p$ along $c$ as
  \begin{equation}
    \label{eq:natural_parallel}
    \P (V^\circ_p):= \P_{Riemann}(V^\circ_p-W_p)+W_q,
  \end{equation}
  where $\P_{Riemann}$ is the Riemannian parallel translation along $c$. We 
	extend the definition to any non-zero vector by using the homogeneity
  property:
  \begin{equation}
    \label{eq:natural_parallel_full}
    \P(V_p):= \F(V_p) \cdot \P \left( \tfrac{1}{\F(V_p)} V_p \right)
  \end{equation}
\end{definition}
In order to simplify the notation, we will use $\P_{R}$ instead of
$\P_{Riemann}$ in the sequel.  From Definition \ref{def:nat_parallel} we
have immediately the following

\begin{property}
  The natural parallel translation \eqref{eq:natural_parallel} is
  homogeneous but, in general nonlinear, and preserves the Randers norm
  function $\F$ associated to the navigation data $(h,W)$.
\end{property}

\begin{proof}
  The homogeneity of the parallel translation $\P$ is guaranteed from the
  construction, however it fails to be additive, that is
  $\P(U_p+V_p) \neq \P(U_p) + \P(V_p)$. Therefore $\P$ is nonlinear in
  general.  On the other hand, let $c$ be a curve from $p$ to $q$, and
  $V^\circ_p$ be a unit vector with respect to the Randers norm function
  \eqref{eq:randers_norm}. Using the relation \eqref{eq:unit_norm} between
  the Finslerian and Riemannian norms, and the fact that the Riemannian
  parallel translation preserves the Riemannian norm, we get
  \begin{equation}
    1 = \F(V^\circ_p)= \norm{V^\circ_p\!-\!W_p}  = 
    \norm{\P_R (V^\circ_p \! - \! W_p)}  =  \F 
    \bigl((\P_R (V^\circ_p \! - \! W_p)+W_q\bigr)  = \F
    \bigl(\P(V^\circ_p)\bigr)
  \end{equation}
  showing that $\P$ preserves the Finslerian norm of unit vectors.  Using
  the homogeneity property of the parallel translation $\P$ and the norm
  function $\F$, we can get that the Finslerian norm of any tangent vector
  is preserved under the natural parallel translation.
\end{proof}
We can observe that for a given curve $c$ joining the points $p$ and $q$,
the value of the parallel translation depends only on the value of $W$ at
$p$ and $q$ but does not depend on the value of $W$ on $c$ between the
endpoints.

\begin{property}
  \label{prop:W_parallel}
  Let $(h, W)$ be a navigation data on the manifold $M$. If $W$ is parallel
  along the curve $c$ with respect to the Riemannian metric, then the
  natural and the Riemannian parallel transports on $c$ coincide.
\end{property}

\begin{proof}
  Indeed, let $c$ be a curve from $p$ to $q$ and suppose that $W$ to be
  parallel with respect to the Riemannian metric $h$ along the curve $c$.
  Then $\P_{R}(W_p)=W_q$.  Let $V^\circ_p \in T_pM$ be a Finslerian unit vector
  at $p$.  Using the linearity of the Riemannian parallel translation
  $\P_{R}$:
  \begin{equation}
    \label{eq:natural_parallel_ct}
    \begin{aligned}
      \P (V^\circ_p) = \P_{R}(V^\circ_p-W_p) +W_q
      = \P_{R}(V^\circ_p) -\P_{R}(W_p)+W_q
      = \P_{R}(V^\circ_p),
    \end{aligned}
  \end{equation}
  showing that $\P$ and $\P_R$ coincide on Finslerian unit vectors. The
  statement follows from the homogeneity property of the parallel
  translations $\P$ and $\P_R$.
\end{proof}

\begin{proposition}
  Let $(h, W)$ be navigation data on the manifold $M$.  The holonomy
  group $\mathcal Hol(\P)$ associated to the natural parallel translation
  is isomorphic to the Riemannian holonomy group $\mathcal Hol
  (\P_{R})$. In particular, the holonomy group of $\mathcal Hol (\P)$ is
  finite dimensional.
\end{proposition}

\begin{proof}
  Taking $p=q$, a straightforward computation shows the one-to-one
  correspondence $\varphi \leftrightarrow \varphi_R$ between the elements
  of the holonomy groups, where
  \begin{displaymath}
    \varphi(V_p)=\varphi_R(V_p)-\mathcal{F}(V_p)\left(\varphi_R(W_p)
      -W_p\right).
  \end{displaymath}
  Substituting $V_p=W_p$,
  \begin{displaymath}
    \varphi_R(W_p)=\frac{\varphi(W_p)-\mathcal{F}(W_p)W_p}{1-\mathcal{F}(W_p)}
  \end{displaymath}
  and the inverse formula is
  \begin{displaymath}
    \varphi_R(V_p)=\varphi(V_p)+\mathcal{F}(V_p)\left( \frac{ \varphi(W_p)
        - \mathcal{F}(W_p) W_p}{1 - \mathcal{F}(W_p)}-W_p \right)
    =\varphi(V_p)+\mathcal{F}(V_p)\frac{\varphi(W_p)-W_p}{1-\mathcal{F}(W_p)}.
  \end{displaymath}
  Finally,
  \begin{alignat*}{1}
    \psi\circ \varphi (V_p)
    & = \psi_R(\varphi(V_p)) - \mathcal{F}(\varphi(V_p))\left(\psi_R(W_p)
      - W_p\right)
    \\
    & = \psi_R\circ \varphi_R(V_p)-\mathcal{F}(V_p) \left(\psi_R\circ
      \varphi_R(W_p)-\psi_R(W_p)\right)-\mathcal{F}(\varphi(V_p))
      \left(\psi_R(W_p)-W_p\right)
    \\
    & = \psi_R\circ \varphi_R(V_p)-\mathcal{F}(V_p)\left(\psi_R\circ
      \varphi_R(W_p)-W_p\right)
  \end{alignat*}
because of $\mathcal{F}(\varphi(V_p))=\mathcal{F}(V_p)$.

\end{proof}

\begin{remark}
  As shown in \cite{Hubicska_Matveev_Muzsnay_2021}, the holonomy groups of
  homogeneous parallel translations arising from a Finsler metric are
  generally infinite-dimensional, with finite-dimensional holonomy groups
  appearing only in special cases.  This makes it particularly interesting
  to identify homogeneous (nonlinear) parallel translations with
  \emph{finite-dimensional holonomy groups}. As demonstrated by the
  previous proposition, natural parallel translations possess this
  distinctive property.
\end{remark}

\bigskip

\section{Geometric quantities associated to the natural parallelism}
\label{sec:4}

As presented in Section \ref{sec:2}, the notion of connection, covariant
differentiation and splitting \eqref{eq:T_h_v} are geometric structures
related to the parallelism.  In this section we derive these geometric
objects associated to the natural parallel translation with respect to the
navigation data $(h, W)$.

\subsection{Horizontal distribution, connection}
\label{subsec:hordistr}

As Remark \ref{rem:parallel} shows, parallel translations correspond to
traveling along horizontal curves. Therefore the horizontal distribution
associated to a parallelism can be obtained by differentiation of parallel
vector fields: the possible tangent directions are horizontal.  We can use
this property to determine the horizontal distribution.

Let $A^k_{ij}=A^k_{ij}(x)$ be the connection coefficients of the
Lévi-Civita connection $\nabla^R$ of the Riemann metric $h$:
\begin{equation}
  \label{eq:riemann_horiz}
  \nabla^{^R}_{\! \frac{\partial}{\partial x^i}}
  {\tfrac{\partial}{\partial x^j}}
  =  A^k_{ij} {\tfrac{\partial}{\partial x^k}}.
\end{equation}
Consider a curve $c_t$ on $M$ and let $\P_{t}$ be the natural parallel
translation along $c_t$, and $(\P_R)_{t}$ be the Riemann parallel translation
with respect to the metric $h$ along the curve $c_t$. Then, from formula
\eqref{eq:natural_parallel} we get that for any Finslerian unit vector
$V_\mathrm{o } \in T_{c_o} M$ the corresponding parallel vector field along
$c_t$ is
\begin{equation}
  \label{eq:parallel_field}
  \P_t(V_0)  = (\P_R)_t(V_0- W_0)+ W(c(t)),
\end{equation}
and its derivative must be horizontal:
\begin{equation}
  \label{eq:parallel_field_der}
  \frac{d}{dt}  \P_t(V_0)  \in \mathcal H_{(c_t, \P_t(V_0))}.
\end{equation}
If the coordinate expression of the parallel vector field
\eqref{eq:parallel_field} in $(x^k, y^k)$ is
$(c_t^k, \P_t (V_0)^k)$, then its derivative is
\begin{equation}
  (c_t^k, \P_t(V_0)^k, \dot{c}_t^k, \tfrac{d}{dt}[\P_t(V_0)]^k), 
\end{equation}
where
\begin{equation}
  \label{eq:deriv_Pt}
    \frac{d}{dt} [\P_t(V_0)]^k
    \! = \! \frac{d}{dt}
    \bigl[(\P_R)_t(V_0 \! - \!W_0)+ W(c_t)\bigr]^k
    \! = \! - A^k_{ij} \, \dot{c}_t^i \, (\P_R)^j_t(V_0- W_0) +
    \frac{\partial W^k}{\partial x^i}\dot{c}_t^i.
\end{equation}
Denoting the connection coefficients of the natural parallelism with
$\Gamma^k_{i}=\Gamma^k_{i}(x,y)$, the horizontal projector associated to
the navigation data is
\begin{equation}
  \label{eq:horiz_proj}
  \h \colon TTM \rightarrow \mathcal H \quad (\subset TTM),
\end{equation}
and the horizontal distribution is generated by the vector fields
\begin{equation}
  \label{eq:natural_horiz}
  \h \Bigl(\frac{\partial}{\partial x^i }\Bigl)  =
  \displaystyle \frac{\partial}{\partial x^i }
  - \Gamma_i^j  \frac{\partial}{\partial y^j }.
\end{equation}
Comparing \eqref{eq:parallel_field_der}, \eqref{eq:deriv_Pt}, and
\eqref{eq:natural_horiz} we get 
\begin{equation}
  \label{eq:Gamma_1}
  \Gamma^k_i(c_t, \P_t(V_0)) \dot{c}_t^i   
  =  A^k_{ij}(c_t) \cdot \dot{c}_t^i \cdot [(\P_R)_t(V_0- W_0)]^j -
  \dot{c}^i_t  \frac{\partial W^k}{\partial x^i}\Big|_{c(t)}.
\end{equation}
The formula holds for any curve $c$ and any parameter $t$, so for any point
$c_t$, and direction represented by the unit vector $\dot{c}_o \in T_xM$. We
also note that for $t=0$ we have $\P_o=(\P_R)_o=id$ the identity
transformation. Therefore at $x=c_o$ we get
\begin{equation}
  \label{eq:Gamma_2}
  \Gamma^k_i(x, V_0)= A^k_{ij}(x)  (V_0^j- W^j(x))
  - \frac{\partial W^k}{\partial x^i}\Big|_x
\end{equation}
for any Finslerian unit vector $V_0$. Moreover, we can extend it by the
homogeneity property to any, not necessarily unit vector: for any
$y\in \T_xM$ we have a Finslerian unit vector $V_0=\frac{1}{\F(x,y)}y$. Therefore
\begin{equation}
  \label{eq:Gamma_3}
  \Gamma^k_i(x, \tfrac{y}{\F(x,y)})=  A^k_{ij}(x) \cdot(\tfrac{1}{\F(x,y)}y^j
  -W^j(x)) -\frac{\partial W^k}{\partial x^i}\Big|_{x},
\end{equation}
or equivalently
\begin{equation}
  \label{eq:Gamma_4}
  \Gamma^k_i(x, y)=  A^k_{ij}(x) \cdot(y^j- \F(x,y) W^j(x))
  - \F(x, y) \frac{\partial W^k}{\partial x^i}\Big|_{x}.
\end{equation}
%Since $V$ can be any tangent vector at $x$ we have 
%\begin{equation}
%  \label{eq:Gamma_5}
 % \Gamma^k_i(x, y)=  A^k_{ij}(x) \cdot 
 % \bigl[ y^j- \F(x,y) \cdot W^j(x) \bigr]
  %- \F(x,y) \cdot \frac{\partial W^k}{\partial x^i}\Big|_{x},
%\end{equation}
Using the simplified notation $A^k_{ij}=A^k_{ij}(x)$ and $\F(y) =\F(x,y)$,
we have
\begin{equation}
  \label{eq:Gamma_6}
  \Gamma^k_i(x,y) =  A^k_{is} \cdot y^s - A^k_{is} \cdot \F(y)\cdot  W^s
  - \F(y)\cdot \frac{\partial W^k}{\partial x^i}.
\end{equation}
The horizontal distribution can be described as the distribution spanned by
the horizontal lifted vector fields: the horizontal lift
$l_{(x,y)}:T_xM \to \mathcal H_{(x,y)}$ of $X \in T_xM$ is
\begin{equation}
  \label{eq:horizont}
  l_{(x,y)}(X)=l_{(x,y)}^{R} (X) +\F (x, y) \cdot \left(\nr_X W\right)^v,
\end{equation}
where $l_{(x,y)}^{R}$ is the Riemannian horizontal lift, and $\nr_X W$ is
the vertical lift of the Riemannian covariant derivative of the wind $W$.
\begin{definition}
  \label{def:natural_connection}
  Let $(h, W)$ be a navigation data on the manifold $M$, and $\h$ be the
  horizontal projector \eqref{eq:horiz_proj} associated to the natural
  parallelism.  The connection
  \begin{equation}
    \label{eq:natural_connection}
    \Gamma:=2\h - \mathrm{Id}.
  \end{equation}
  is called the \emph{natural connection}. The coefficients of the natural
  connection are the functions given in \eqref{eq:Gamma_6}.
\end{definition}

\bigskip

\subsection{Covariant derivative associated with natural paralellism}
\label{subsec:covder}

A vector field $V_t$ along the curve $c_t$ is parallel if its derivative is
the horizontal lift of $\dot{c}(t)$ along $V_t$:
\begin{equation}
  \label{eq:lift_V_t}
  \dot{V}_t = l_{(c,V_t)} (\dot{c}),
\end{equation}
or equivalently, if its covariant derivative is vanishing. From
\eqref{parallelcurve} one can get the formula for the covariant derivative
\begin{equation}
  \label{eq:nabla_vt}
  \frac{\nabla V_t}{dt} =  \Bigl( \frac{d V^k_t}{dt}
  + \Gamma^k_i(c,V)\dot{c}^i\Bigr) \frac{\partial}{\partial x^k}. 
\end{equation}
Formula \eqref{eq:nabla_vt} allows us to introduce the (nonlinear)
covariant derivative of a vector field $V$ with respect to $X$:
\begin{equation}
  \nabla_X V =   X^i \Bigl(\frac{\partial V^k}{\partial
    x^i} + \Gamma^k_i(x,V) \Bigr) \frac{\partial}{\partial x^k}.
\end{equation}
The map 
\begin{equation}
  \label{eq:nabla}
  \nabla : \X{M} \times \X{M} \to \X{M} \qquad  (X, Y) \to \nabla_XY
\end{equation}
is $C^\infty(M)$ linear in its first variable, but only $\R$ homogeneous in
the second variable.  Indeed, $\nabla$ does not satisfy the additivity
property in the second variable. The parallelism can be formulated in terms
of the (nonlinear) covariant derivative as follows.
\begin{property}{}
  The vector field $V(t)$ along the curve $c$ is parallel if and only if
  $\nabla_{\dot{c}}V=0$.
\end{property}

\begin{property}
  The coordinate invariant expression of the non-linear covariant
  differentiation \eqref{eq:nabla} associated to the nonlinear parallelism
  is
  \begin{equation}
    \label{eq:covariant_natural}
    \nabla_XY = \nr_XY - \F(Y) \! \cdot \! \nr_X W,
  \end{equation}
  where $\nr$ is the Levi-Civita connection of the Riemannian metric $h$.
\end{property}

\begin{remark}
  Formula \eqref{eq:covariant_natural} shows that if the wind $W$ is
  parallel along a curve, then the Riemannian parallel translation $\P_R$
  and the natural parallel translation $\P$ along $c$ coincide. Moreover,
  $W$ is parallel with respect to the Riemannian metric if and only if
  $ \nabla \equiv \nr \!$.
\end{remark}

Investigating the natural parallelism and the associated covariant
differentiation we get the following proposition.

\begin{proposition}
  The covariant differentiations along the integral curves of the vector
  field $W$ coincide if and only if the integral curves of $W$ are
  Riemannian geodesics.
\end{proposition}
\begin{proof}
  Substituting $X=W$ in formula \eqref{eq:covariant_natural} we have the
  same covariant derivatives along the integral curves of the vector field
  $W$ if and only if $\nabla^R_{\dot{c}} W=0$, that is the integral curves
  of $W$ are Riemannian geodesics.
\end{proof}

\begin{proposition}
  The integral curves of $W$ are pre-geodesics (resp.~geodesics) of
  $\nabla$ if and only if they are pre-geodesics (resp.~geodesics) of
  $\nr$.
\end{proposition} 
\begin{proof}
  Substituting $X=Y=W$ in formula \eqref{eq:covariant_natural} we have that
  \begin{equation}
    \label{eq:x_y_w}
    \nabla_W W=\left(1-\F(W)\right)\nabla^R_W W.
  \end{equation}
  Since $W$ is one of the navigation data, it follows that
  \begin{equation}
    \label{eq:F_W}
    \F(W)=\frac{|W|}{1+|W|}< 1,
  \end{equation}
  and the acceleration vector fields $\nabla_{\dot{c}}W$ and
  $\nabla^R_{\dot{c}} W$ are proportional to $\dot{c}=W\circ c$ at the same
  time as follows:
  \begin{equation}
    \label{eq:pregeod}
    \nabla_{\dot{c}} W (t)=\varphi(t) \dot{c}(t) \qquad \Leftrightarrow
    \qquad \nabla^R_{\dot{c}} W (t)=\rho(t) \dot{c}(t),
  \end{equation}
  with
  \begin{equation}
    \label{eq:rho_varphi}
    \rho(t)=\frac{\varphi(t)}{1-\F (\dot{c}(t))}.
  \end{equation}
  As \eqref{eq:pregeod} shows, the integral curves of $W$ are pre-geodesics of
  the natural connection $\nabla$ if and only if they are pre-geodesics of the
  connection $\nabla^R$ and, in case of $\varphi(t)=\rho(t)=0$, the integral
  curves of $W$ are geodesics of $\nabla$, if and only if they are geodesics of
  $\nabla^R$.

\end{proof}

\subsection{Torsion}

In his paper \cite{Grifone_1972} J.~Grifone developed the connection
theory in terms of the Frölichher-Nijenhuis calculus. He introduced the
notion of weak, and strong torsions of a connection denoted by $t$ and $T$,
respectively. If the connection is homogeneous, then $t=0$ if and only if
$T=0$ (see \cite[Corollaire 1.56]{Grifone_1972}).  Since in the case of
natural parallelism the associated connection is homogeneous, the strong
torsion is vanishing if and only if the weak torsion is zero, and we call
simply by torsion the vector valued 2-form defined by
\begin{equation}
  \label{eq:def_torsion}
  t=\frac{1}{2}[J, \Gamma].
\end{equation}
Its components are given by
\begin{equation}
  \label{eq:torsion_1}
  t^k_{ij}=\frac{\partial \Gamma^k_j}{\partial y^i}
  -\frac{\partial \Gamma^k_i}{\partial y^j}.
\end{equation}
The torsion of the natural connection is nonzero in general. Indeed, the
derivatives of the connection coefficients with respect to vertical
directions are
\begin{equation}
  \label{eq:torsion_2}
  \frac{\partial \Gamma^k_i}{\partial y^j}(x, y)= A^k_{ij} 
  -\F_{y^j} A^k_{is} W^s - \F_{y^j} \frac{\partial W^k}{\partial x^i},
\end{equation}
where $\F_{y^i}:=\partial_{y^i} \F$.  It follows that the torsion
\begin{equation}
  \label{eq:torsion_3}
  t^k_{ij}= (\F_{y^i} A^k_{js}-\F_{y^j} A^k_{is}) W^s
  +  \F_{y^i} \frac{\partial W^k}{\partial x^j}
  - \F_{y^j} \frac{\partial W^k}{\partial x^i},
\end{equation}
or in a coordinate-free way
\begin{equation}
  \label{eq:torsion}
  t =  d_J\F \wedge \nr W,
\end{equation}
where the right side involves the simplified notation of the vertical
lifted vector valued 1-form
\begin{equation}
  \bigl[ \nr W \bigr] (X) := (\nr_X W)^v, 
\end{equation}
and the wedge product gives a vector valued 2-form
defined as
\begin{equation}
  \bigl[d_J\F \wedge \nr W \bigr] (X,Y)=d_J \F(X)
  \cdot \nr W (Y) -d_J \F(Y) \cdot  \nr W (X),
\end{equation}
or equivalently
\begin{equation}
  [d_J\F \wedge \nr W] (X,Y) =JX(F) \cdot (\nr_Y W)^v
  -JY(F)\cdot (\nr_X W)^v.  
\end{equation}
From \eqref{eq:torsion} it follows that
\begin{equation}
  \label{eq:torsion_cond}
  t  \equiv 0 \quad \Longleftrightarrow \quad d_J\F \wedge \nr W = 0.
\end{equation}
The geometric characterization of \eqref{eq:torsion_cond}, the vanishing of
the torsion, is given by Proposition \ref{prop:t=0_W} and Theorem
\ref{thm:t=0_W_randers} below.

\begin{proposition}
  \label{prop:t=0_W}
  The  torsion of the natural connection is identically zero if and
  only if $W$ is parallel with respect to the Riemannian metric, that
  is $\nr W=0$.
\end{proposition}
In order to prove the proposition we need the following lemma. 

\begin{lemma}
  \label{torsionlemma}
  Let $\rho \in \Psi^1(M)$ be a vector valued 1-form on the base manifold
  $M$. Then
  \begin{equation}
    \label{eq:torsion_lemma}
    \F_{y^i}\rho_j^k\circ \pi=\F_{y^j}\rho_i^k\circ \pi
  \end{equation}
  is satisfied if and only if $\rho_j^k=0$.
\end{lemma}
\begin{proof}
  Let us suppose that \eqref{eq:torsion_lemma} is satisfied. Then, using
  the homogeneity property of the Finslerian metric function, the
  contraction by $y^i$ gives that
  \begin{equation}
    \F\rho_j^k\circ \pi=\F_{y^j}y^i\rho_i^k\circ \pi,
  \end{equation}
  or equivalently
  \begin{equation}
    \rho_j^k\circ \pi=\frac{1}{\F}\F_{y^j}y^i\rho_i^k\circ \pi.
  \end{equation}
  Differentiating by the variable $y^l$ we have
  \begin{alignat*}{1}
    0&=-\frac{1}{\F^2}\F_{y^l}\F_{y^j}y^i\rho_i^k\circ
       \pi+\frac{1}{\F}\F_{y^j y^l}y^i\rho_i^k\circ
       \pi+\frac{1}{\F}\F_{y^j}\rho_l^k\circ \pi
    \\
     &=
       -\frac{1}{\F^2}\F_{y^l}\F_{y^j}y^i\rho_i^k\circ
       \pi+\frac{1}{\F}\F_{y^j y^l}y^i\rho_i^k\circ
       \pi+\frac{1}{\F}\F_{y^j}\left(\frac{1}{\F}\F_{y^l}y^i\rho_i^k\circ
       \pi\right)
       =\frac{1}{\F}\F_{y^j y^l}y^i\rho_i^k\circ \pi.
  \end{alignat*}
  Since $\F \cdot \F_{y_j y_l}$ is the component of the angular metric tensor
  (the restriction of the Riemann-Finsler metric to the indicatrix bundle),
  it follows that $y^i\rho_i^k\circ \pi=0$ and, by differentiating with
  respect to the variable $y^j$, we have the vanishing ot the components of
  the one-form. The converse of the statement is trivial.
\end{proof}

\begin{proof}[Proof of Proposition \ref{prop:t=0_W}]
  Using equations \eqref{eq:torsion_1}, \eqref{eq:torsion_2}, and
  \eqref{eq:torsion_3}, we have the result by Lemma \ref{torsionlemma}
  under the choice $\rho(X)=\nr_X W$.
\end{proof}

\begin{theorem}
  \label{thm:t=0_W_randers}
  The  torsion of the natural connection is identically zero if and
  only if the Randers metric $\F$ is Berwaldian.
\end{theorem}

\begin{proof}
  Let us suppose that the  torsion of the natural connection is
  identically zero. Then, by Proposition \ref{prop:t=0_W}, the vanishing of
  the  torsion is equivalent to the parallelism of the vector field
  $W$. If $W$ is parallel with respect to the Riemannian metric $h$, then
  the natural parallel translation and the Riemannian parallel translation
  with respect to the Riemannian metric $h$ coincide. The natural parallel
  translation preserves the Finlserian norm function $\F$. So does the
  Riemannian parallel translation with respect to the Riemannian metric
  $h$. Therefore $\F$ is Berwaldian and its canonical connection is
  $\nabla^R$ as a torsion-free linear connection such that the parallel
  transport preserves the Finslerian metric function $\F$.

  To prove the converse, let us suppose that the Randers metric $\F$ is
  Berwaldian. We set the following notations:
  \begin{equation}
    \label{eq:randers}
    \F(x,y)=\sqrt{\alpha_{ij}(x)y^iy^j}+\beta_i(x)y^i,
  \end{equation}
  where
  \begin{equation}
    \alpha_{ij}=\frac{1}{\lambda}h_{ij}+\beta_i\beta_j,
    \ \ \textrm{and} \ \ \beta_i=-\frac{1}{\lambda}h_{ik}W^k.
  \end{equation}
  Using the abbreviation $\gamma_{ij}=h_{ij}/\lambda$, we have 
  \begin{equation}
    \alpha_{ij}=\gamma_{ij}+\beta_i\beta_j,
    \ \ \textrm{and} \ \ \beta_i=-\gamma_{ik}W^k.
  \end{equation}
  In particular,
  \begin{equation}
    \alpha^{ij}=\gamma^{ij}-\frac{1}{1+\beta^s\beta_s}\beta^i \beta^j,
  \end{equation}
  where $\beta^i=\gamma^{ik}\beta_k$.  It is known (see for example
  \cite[Theorem 1]{Vincze2015}) that $\F$ is Berwaldian if and only if
  $\nabla^{\alpha} \beta =0$, where $\nabla^{\alpha}$ is the Lévi-Civita
  connecton of the Riemannian metric $\alpha$. In terms of local
  coordinates we have 
  \begin{equation}
    \label{eq:covconst}
    \partial_i \beta_j-\Gamma^k_{ij} \beta_k=0,
  \end{equation}
  where the functions $\Gamma^k_{ij}$ are the Christoffel symbols of
  $\nabla^{\alpha}$. In particular,
  \begin{equation}
    \label{eq:d_beta}
    \partial_i \beta_j=\partial_j \beta_i.
  \end{equation}
  Since $\beta$ is a covariant constant one-form, it follows that its dual
  vector field is of constant length with respect to the Riemannian metric $\alpha_{ij}$:
  \begin{equation}
    \textrm{const}=\alpha^{ij}\beta_i\beta_j=\left (\gamma^{ij}
      -\frac{1}{1+\beta^s \beta_s} \beta^i \beta^j\right)\beta_i\beta_j
    =\frac{\beta^s\beta_s}{1+\beta^s\beta_s}
    =\frac{|W|_{\gamma}^2}{1+|W|_{\gamma}^2}=|W|^2.
  \end{equation}
  Therefore the conformal factor $\lambda$ between the Riemannian metrics
  $\gamma_{ij}$ and $h_{ij}$ is constant and
  $\nabla^R=\nabla^\gamma$. Using that
  \begin{equation}
    \alpha^{km}\beta_k=\left(\gamma^{km}-\frac{1}{1+\beta^s\beta_s}\beta^k
      \beta^m\right)\beta_k
    =\frac{1}{1+\beta^s\beta_s}\beta^m
    =\frac{1}{1+\beta^s\beta_s}\gamma^{km}\beta_k
  \end{equation}
  we have
  \begin{alignat*}{1}
    0=&\partial_i \beta_j-\Gamma_{ij}^k \beta_k=\partial_i
        \beta_j-\tfrac{1}{2}\alpha^{km}\left(\partial_i\alpha_{jm}
        +\partial_j\alpha_{im}-\partial_m \alpha_{ij}\right)\beta_k
    \\
    =&\partial_i \beta_j-\tfrac{1}{2}\tfrac{1}{1+\beta^s\beta_s} \gamma^{km}
       \left(\partial_i\gamma_{jm}+\partial_j\gamma_{im}
       -\partial_m    \gamma_{ij}\right)\beta_k
    \\
      &- \tfrac{1}{2}\tfrac{1}{1+\beta^s\beta_s}\gamma^{km}
        \left(\beta_m\left(\partial_i\beta_j +\partial_j\beta_i\right) +
        \beta_j\left(\partial_i \beta_m - \partial_m  \beta_i\right)
        + \beta_i\left(\partial_j \beta_m-\partial_m  \beta_j\right)\right)
        \beta_k\stackrel{\eqref{eq:d_beta}}{=}
    \\
    = & \partial_i \beta_j - \tfrac{1}{1+\beta^s\beta_s}A_{ij}^k\beta_k
        - \tfrac{\beta^s\beta_s}{1+\beta^s\beta_s}\partial_i\beta_j 
        = \frac{1}{1+\beta^s\beta_s}
        \left(\partial_i\beta_j-A_{ij}^k\beta_k\right).
  \end{alignat*}
  Therefore the one-form $\beta$ is parallel with respect to $h$, that is
  its dual vector field is covariant constant: $\nabla^R W=0$ and,
  according to Proposition \ref{prop:t=0_W}, the torsion is identically
  zero.
\end{proof}

\begin{theorem}
  The vector field $W$ is of constant length with respect to the Riemannian
  metric $h$ if and only if the Randers metric $\F$ is Wagnerian.
\end{theorem}

\begin{proof}
  Using the notations in the proof of Theorem \ref{thm:t=0_W_randers}, it is known (see for example \cite[Theorem 2]{Vincze2015}) that $\F$ is   Wagnerian if and only if the dual vector field of $\beta$ is of constant length with respect to
  the Riemannian metric $\alpha$. The statement
  follows immediately by the computation
  \begin{displaymath}
    \alpha^{ij}\beta_i\beta_j=\left (\gamma^{ij} - \frac{1}{1+\beta^s
        \beta_s}\beta^i \beta^j\right)\beta_i\beta_j
    =\frac{\beta^s\beta_s}{1+\beta^s\beta_s}
    =\frac{|W|_{\gamma}^2}{1+|W|_{\gamma}^2}=|W|^2.
  \end{displaymath}
  The parallel transport preserving the Finslerian metric is given by the
  linear connection
  \begin{displaymath}
    \nabla_X Y=\nabla^{\alpha}_X Y+\frac{\alpha(\nabla^{\alpha}_X
      \beta^{\sharp}, Y)\beta^{\sharp} -\alpha(Y, \beta^{\sharp})
      \nabla^{\alpha}_X \beta^{\sharp}}{|\beta^{\sharp}|_{\alpha}^2},
  \end{displaymath}
  where $\nabla^{\alpha}$ is the Lévi-Civita connection of the Riemannian
  metric $\alpha$ and the sharp operator is taken with respect to
  $\alpha$. It can be easily seen that $\nabla \alpha=0$ and
  $\nabla \beta^{\sharp}=0$ provided that the dual vector field is of
  constant length.  Therefore, we have both $\nabla \alpha=0$ and
  $\nabla \beta=0$, that is, the Randers metric is invariant under the
  parallel transports with respect to $\nabla$.
  
\end{proof}

\bigskip

\section{Autoparallel curves, natural spray}
\label{sec:5}

Let $(h,W)$ be navigation data on the manifold $M$. The curve
$c\colon I \to M$ is autoparallel with respect to the natural parallelism
if the velocity vector field $\dot{c}$ along $c$ is parallel with respect
to the parallelism introduced in Definition \ref{def:nat_parallel}.  Using
the covariant derivative $\nabla$ introduced in section \ref{subsec:covder}
we get that $c$ is an autoparallel curve if and only if
\begin{equation}
  \label{eq:autop}
  \nabla_{\!\dot{c}} \, \dot{c} =0.
\end{equation}
According to the expression \eqref{eq:covariant_natural} of $\nabla$ in
terms of the Levi-Civita connection $\nr$ of the Riemannian metric $h$ and
the vector field $W$ we get that $c$ is autoparallel with respect to the
natural parallel structure if and only if it satisfies
\begin{equation}
  \label{eq:naural_geod}
  \nr_{\!\dot{c}} \, \dot{c}- \F(\dot{c}) \! \cdot \!
  \nr_{\!\dot{c}} \, W=0
\end{equation}
In a local coordinate system we obtain the second order differential
equation
\begin{equation}
  \label{eq:geod_eq_2}
  \ddot{c}^k + \dot{c}^i 
  \left( A^k_{ij} \dot{c}^j - \F(\dot{c}) A^k_{ij} W^j
    - \F(\dot{c}) \tfrac{\partial W^k}{\partial x^i}
  \right)  =0,
\end{equation}
where $A^k_{ij}$ are the Christoffel symbols of the Levi-Civita connection
$\nr$. The spray
\begin{equation}
  \label{eq:nat_spray}
  S = y^k \frac{\partial}{\partial x^k}
  -2G^k(x,y) \frac{\partial}{\partial y^k}
\end{equation}
corresponding to the system \eqref{eq:geod_eq_2} has coefficients 
\begin{equation}
  \label{eq:G_k}
  G^k(x,y) =  \tfrac{1}{2} \left( A^k_{ij} y^i y^j - \F(x,y) y^i A^k_{ij}
    W^j - \F(x,y)y^i \tfrac{\partial W^k}{\partial x^i}   \right).
\end{equation}
\begin{definition}
  The spray \eqref{eq:nat_spray} with coefficients \eqref{eq:G_k}
  corresponding to the natural parallelism will be called the \emph{natural
    spray}, and the connection
  \begin{equation}
    \label{eq:nat_sym_con}
    \overline{\Gamma}:=[J,S]
  \end{equation}
  generated by the
  natural spray will be called the \emph{natural symmetric connection}.
\end{definition}

\begin{remark}
  The natural connection \eqref{eq:natural_connection} and the natural
  symmetric connection \eqref{eq:nat_sym_con} are  different in general. Indeed, the torsion of the natural symmetric connection is
  identically zero but, in general, the torsion of the natural connection
  is not.
\end{remark}

\begin{lemma} 
  \label{projectivelemma} 
  Equation
  \begin{equation}
    \label{eq:lemma}
    y^i\rho_i^k\circ \pi=\varphi(x,y) y^k
  \end{equation}
  is satisfied for a zero homogeneous function $\varphi$ on the tangent
  manifold and a one-form $\rho$ on the base manifold if and only if
  $\varphi(x,y)=\varphi(x)$ and $\rho_j^k(x)=\varphi(x) \delta_j^k.$
\end{lemma}

\begin{proof}
  First of all we prove that the zero homogeneous function $\varphi$
  depends only on the position. Differentiating with respect to the
  variable $y^l$ we have
  \begin{displaymath}
    \rho^k_l\circ \pi=\varphi_{y^l} y^k+\delta_l^k \varphi.
  \end{displaymath}
  Taking $k=l$
  \begin{displaymath}
    \rho^k_k\circ \pi=\varphi_{y^k} y^k+\delta_k^k \varphi,
  \end{displaymath}
  where $y^k \varphi_{y^k}=0$ because $\varphi$ is zero
  homogeneous. Therefore
  \begin{displaymath}
    \varphi=\frac{\rho_k^k\circ \pi}{n}\ \ \Rightarrow\ \
    \varphi(x,y)=\varphi(x)
  \end{displaymath}
  depends only on the position. Finally,
  \begin{displaymath}
    y^i\rho_i^k\circ \pi=\varphi\circ \pi y^k\ \ \Rightarrow \ \
    \rho_j^k(x)=\varphi(x)\delta_j^k
  \end{displaymath}
  by differentiating with respect to the variable $y^j$. (The converse of
  the statement is trivial.)
\end{proof}

\begin{definition}
  The vector field $W$ is called a \emph{concircular vector field} with
  respect to the Riemannian metric $h$ if there exist a function
  $\varphi \in C^\infty(M)$ such that
  \begin{equation}
    \label{eq:concircular_W}
    \nabla^R_X W=\varphi \, X, 
  \end{equation}
  for all $X \in \X{M}$, where $\nabla^R$ denotes the Levi-Civita
  connection of $h$. The function $\varphi$ is called the potential
  function of $W$.
\end{definition}

\begin{proposition}
  \label{prop:S_S_rieman_proj_rel}
  Let $(h, W)$ be a navigation data.   
  The natural spray $S$ is projectively related to the quadratic spray
  $S^R$ corresponding to the geodesic equations of the Riemannian
  connection $\nr$ if and only if $W$ is a concircular vector field with
  respect of $h$.
\end{proposition}

\begin{proof}
  The sprays $S$ and $S^R$ are projectively related iff their paths
  coincide as point sets, that is they coincide up to reparametrizations.
  It is well-known that the equivalent condition is the existence of a
  positively one-homogeneous function $P$ such that $S=S^R + P \mathcal C$.
  The local expression of the quadratic spray $S^R$ of the linear
  connection $\nr$ is
  \begin{equation}
    \label{eq:nat_spray01}
    S^R = y^k \frac{\partial}{\partial x^k}
    -A^k_{ij}(x)y^iy^j \frac{\partial}{\partial y^k}.
  \end{equation}
  Therefore, using equations \eqref{eq:nat_spray}, \eqref{eq:G_k} and
  \eqref{eq:nat_spray01}, the projective equivalence gives that there
  exists a positively 1-homogeneous function $P(x,y)$ such that
  \begin{equation}
    \label{eq:proj_eq}
    \F(x,y)y^i\rho^k_i(x)=P(x,y)y^k,
  \end{equation}
  where $\rho(X)=\nabla^R_X W$.  Applying Lemma \ref{projectivelemma} with
  $\varphi=P/\F$, it follows that the sprays are projectively related if
  and only if $P(x,y)=\varphi(x)\F(x,y)$, that is the projective factor is
  conformally related to the Randers metric $\F$, and
  $\rho_j^k(x)=\varphi(x)\delta_j^k$.

\end{proof}

\begin{example}
  \label{ex:euc_radial_wind}
  Let $M$ be the interior of the Euclidean unit ball in $\mathbb{R}^n$
  equipped by the standard Euclidean inner product $h_{ij}=\delta_{ij}$. If
  \begin{equation}
    W=-x^i\frac{\partial}{\partial x^i}
  \end{equation}
  then we have that $\nabla^{R}_X W=-X$ and the natural spray is
  \begin{equation}
    \label{Funkspray}
    S=y^i\frac{\partial}{\partial x^i}-\F(x,y) \mathcal C,
  \end{equation}
  where $C$ is the so-called
  Liouville vector field and the Randers metric $\F$ coincides the Funk
  metric \cite{Okada_1983}. Therefore, the natural spray \eqref{Funkspray} is
  metrizable as the canonical spray of the Funk metric.
\end{example}
Example \ref{ex:euc_radial_wind} is particularly interesting, since the
geodesic structure obtained from the natural parallelism is
metrizable. This is not always the case,  as the next example shows.

\begin{example}
  Let $h$ be the Euclidean metric on $\R^2$, defined by
  $h_{ij}=\delta_{ij}$, and let the wind $W$ be given by the infinitesimal
  rotation
  \begin{equation}
    \label{eq:_rotation}
    W =x_1\frac{\partial}{\partial x_2}-x_2\frac{\partial}{\partial x_1}.
  \end{equation}
  One can show that the Lie brackets $[\delta_1,\delta_2]$, and
  $\bigl[[\delta_1,\delta_2],\delta_2 \bigr]$ of the horizontal vector
  fields \eqref{eq:natural_horiz} give independent vertical directions,
  therefore the distribution generated by horizontal vector fields:
  \begin{equation}
    \label{eq:14}
    \mathcal D_{\mathcal H}:= \bigl\langle \mathcal H
    \bigr \rangle_{Lie}  \!
    =\emph{Span}\big\{\delta_i, \, [\delta_{i_1},\delta_{i_2}], \,
    \bigl[\delta_{i_1}, [\delta_{i_2},\delta_{i_3}]
    \bigr] \dots  \big\},
  \end{equation}
  called the holonomy distribution, is $4$-dimensional: it contains the
  whole horizontal and vertical space, that is
  \begin{equation}
    \label{eq:D_H}
    \mathcal D_{\mathcal H}=TTM.
  \end{equation}
  If the associated natural spray were Finsler metrizable, then the
  corresponding Finsler norm function $F$ would be a holonomy invariant
  function \cite{Elgendi_Muzsnay_2017}. Therefore, for any
  $X \in \mathcal D_{\mathcal H}$ one should have $\mathcal L_X F
  =0$. However, from \eqref{eq:D_H}, it follows that $F$ must be a constant
  function, which leads to a contradiction. This shows that, in this case,
  the natural spray is not Finsler metrizable.
\end{example}

It is a natural problem to characterize navigation data $(h,W)$ for which
the geometric structures — such as the parallel translations or the
autoparallel curves and sprays — associated with the natural parallelism
and the Randers-type Finsler metric $\F$ coincide. The simplest situation
occurs when the parallel translations coincide. We have the following

\begin{property}
  Let $(h,W)$ be a navigation data on a manifold $M$. The parallel
  translations with respect to the associated Randers metric $\F$ and the
  natural parallel translations \eqref{eq:natural_parallel} coincide if and
  only if $W$ is parallel with respect to $h$.
\end{property}

\begin{proof}
  Let us suppose that the parallel translations coincide.  Using Grifone's
  terminology, the torsion of the connection associated with $\F$, and
  therefore the torsion of the natural connection \eqref{eq:torsion_1}
  vanishes.  From Proposition \ref{prop:t=0_W} we get that $W$ is parallel
  with respect to the Riemann metric $h$, and using Property
  \ref{prop:W_parallel} we find that all three parallel translations, that
  is the natural, the Riemannian, and the Finslerian (which is actually of
  Berwald type -- see Theorem \ref{thm:t=0_W_randers}) coincide.
\end{proof}

\begin{remark}
  It is important to note that the autoparallel curves (resp.~sprays) of
  the Randers metric $\F$ and those arising from the natural parallel
  translations associated with navigation data may coincide, even when the
  parallel translations themselves do not. A concrete illustration of this
  phenomenon is given in \ref{ex:euc_radial_wind}, where the navigation
  data $(h,W)$ involves a vector field $W$ that is not parallel with
  respect to $h$.  The two parallel transports are clearly different since
  the holonomy group of the natural parallel translation is
  finite-dimensional, whereas the holonomy of $\F$, the Funk metric, is
  known to be infinite-dimensional \cite{Hubicska_Muzsnay_2020a}.  In the
  proposition below, we characterize the specific cases in which the
  natural parallelism and the Randers metric associated with a navigation
  problem yield the same geodesic structure.
\end{remark}

\begin{proposition}
  \label{prop:S_S_randers}
  The natural spray \eqref{eq:nat_spray} and the geodesic spray of the
  Randers metric $\mathcal F = \alpha + \beta$ associated with the
  navigation data $(h,W)$ coincide if and only if $W$ is a concircular
  vector field with respect to the Riemannian metric $h$.
\end{proposition}

\begin{proof}
  First, suppose that $W$ is a concircular vector field with respect to the
  Riemannian metric $h$, and let $\tilde{S}$ denote the geodesic spray of
  the Randers metric $\F=\alpha+\beta$. Using \cite[Formula
  2.1]{Cheng_Shen_2009}, the spray coefficients of $\tilde{S}$ are given
  by
  \begin{displaymath}
    \tilde{G}^i=\tfrac{1}{2}y^j y^k A_{kj}^i \! +  \! R_0 y^i \! + \!
    \frac{R_{00}}{2}W^i \! - \! \frac{\F^2}{2}\left(S^i \! + \! R^i \! - \!
      R W^i\right) \!  \! - \! \F\left(S_0^i \! + \! \frac{R}{2}y^i \! + \!
      R_0 W^i\right) \! - \! \frac{R_{00}}{2\F}y^i,
  \end{displaymath}
  where $R_{ij}$ and $S_{ij}$ denote the components of the tensors
  \begin{displaymath}
    \mathcal R(X, Y)=\tfrac{1}{2}\left(h(\nabla^R_X W, Y) \! + \! h(X,
      \nabla^R_Y W)\right),  \quad 
    \mathcal S(X,Y)=\tfrac{1}{2}\left(h(\nabla^R_X W, Y) \! - \! h(X,
      \nabla^R_Y W)\right), 
  \end{displaymath}
  respectively. For any $(0,2)$ tensor $\mathcal T$ with tensor components
  $T_{ij}$ we use the notation
  \begin{subequations}
    \label{eq:components}
    \begin{equation}
      T_j=W^i T_{ij}, \quad T=W^j T_j, \quad  T^i=h^{ij}T_j,
      \quad
      T^i_j=h^{il}T_{lj},
    \end{equation}
    \begin{equation}
      T_0=y^i T_i, \quad T_0^i=y^jT_j^i, \quad T_{00}=y^i y^j T_{ij}.
    \end{equation}
  \end{subequations}
  The coefficients \eqref{eq:G_k} of the natural spray can be rewritten as
  \begin{equation}
    \label{alteq:G_k}
    G^i=\frac{1}{2}y^j y^k A_{kj}^i-\frac{\F}{2}\left(R_0^i+S_0^i\right).
  \end{equation}
  From \eqref{eq:concircular_W} we obtain
  \begin{equation}
    \label{eq:R_1}
    R_{ij}=\varphi h_{ij}, \quad 
    R_{00}=\varphi h^2, \quad R^i_0=\varphi y^i, \quad  S_{ij}=0.    
  \end{equation}
  Therefore, applying \cite[Formula 2.4]{Cheng_Shen_2009} with
  $c=-\varphi/2$, we obtain
  \begin{equation}
    \tilde{G}^i=\frac{1}{2}y^j y^k A_{kj}^i-\frac{\F}{2} \varphi y^i
    =\frac{1}{2}y^j y^k
    A_{kj}^i-\frac{\F}{2}R_0^i\stackrel{\eqref{alteq:G_k}}{=}G^i.
  \end{equation}
  Thus, the spray coefficients of $S$ and $\tilde{S}$ coincide.
  
  Conversely, suppose that the natural spray $S$ and the geodesic spray
  $\tilde S$ of the Randers metric coincide. Then
  \begin{equation}
    \label{eq:G_G}
    G^i=\tilde{G},
  \end{equation}
  and consequently, 
  \begin{equation}
    \label{eq:difference}
    G^i(v)-G^i (-v)=\tilde{G}^i(v)-\tilde{G}^i(-v).
  \end{equation}
  Since the quadratic terms in the variable $v$ clearly cancel, the left
  side of \eqref{eq:difference} becomes
  \begin{equation}
    \label{eq:left_side}
    G^i(v)-G^i (-v)=-\frac{R_0^i+S_0^i}{2}(v)\left(\F(v)+\F(-v) \right)
    =-\alpha(v)\left(R_0^i+S_0^i\right)(v).
  \end{equation}
  Similarly, the right side of \eqref{eq:difference} is
  \begin{equation}
    \label{eq:right_side}
    \tilde{G}^i(v) \! - \! \tilde{G}^i(-v)= A \bigl(\F^2(v)  \! -  \!
    \F^2(-v)\bigr) 
    \! - \! B  \left(\F(v)  \! +  \! \F(-v)\right)
    \! - \! C  \left(\frac{1}{\F(v)}  \! +  \! \frac{1}{\F(-v)}\right) \! ,  
  \end{equation}
  where
  \begin{enumerate}[label=(\roman*)]
  \item $A=\frac{1}{2}\left(S^i+R^i-R W^i\right)$ does not depend on $y$,
    \vspace{3pt}
  \item $B=S_0^i+\frac{1}{2} Ry^i + R_0 W^i$ is linear in the variable $y$,
    \vspace{3pt}
  \item $C=\frac{1}{2} R_{00}y^i$ is a cubic term in the variable $y$.
  \end{enumerate}
  Since
  \begin{displaymath}
    \F^2(v) \! - \! \F^2 (-v) \! = \! 4 \alpha(v)\beta(v), \quad \F(v) \! +
    \! \F(-v) \! = \! 2\alpha(v), 
    \quad 
    \frac{1}{\F(v)} \! + \! \frac{1}{\F(-v)} \! = \!
    \frac{2\alpha(v)}{\alpha^2(v)  \! -  \!  \beta^2(v)},
  \end{displaymath}
  dividing \eqref{eq:difference} by the common factor $\alpha(v)$, and
  comparing the remaining terms by using \eqref{eq:left_side} and
  \eqref{eq:right_side} we can obtain that
  \begin{math}
    \displaystyle \frac{R_{00}}{\alpha^2-\beta^2}y^i
  \end{math}
  must be linear in the directional variable. It follows that the function
  \begin{equation}
    \label{eq:m}
    m:=\frac{R_{00}}{\alpha^2-\beta^2}
  \end{equation}
  does not depend on the directional variable. Hence
  \begin{equation}
    \label{eq:R_2}
    R_{00}=\varphi h^2,  \qquad  R_{ij}=\varphi  h_{ij},
  \end{equation}
  for some function $\varphi \in C^\infty(M)$. Applying \cite[Formula
  2.4]{Cheng_Shen_2009} with $c=-\varphi/2$, the coefficients of the
  Randers spray $\tilde{S}$ reduce to 
  \begin{equation}
    \label{eq:tilde_G_reduced}
    \tilde{G}^i=\frac{1}{2}y^j y^k A_{kj}^i-\F S^i_0-\frac{1}{2}\F^2 S^i
    - \frac{1}{2}\varphi {\F}y^i.
  \end{equation}
  Using \eqref{eq:G_G} together with \eqref{alteq:G_k} and
  \eqref{eq:tilde_G_reduced}, we obtain
  \begin{equation}
    \frac{1}{2}\left(R_0^i+S_0^i\right)= S^i_0 + \frac{1}{2}\F S^i +
    \frac{1}{2} \varphi \, y^i.
  \end{equation}
  Since $\F $ is the only nonlinear term in this equation, it follows that
  $S^i=0$, and therefore
  \begin{equation}
     \varphi \, y^i= R_0^i-S_0^i.
  \end{equation}
  Differentiating with respect to $y^j$ yields
  \begin{equation}
    \varphi \delta_j^i= R_j^i-S_j^i  \ \ \Rightarrow \ \ \varphi h_{ij}=
    R_{ij}-S_{ij}.
  \end{equation}
  This implies \eqref{eq:concircular_W}, and thus $W$ is a concircular
  vector field with potential function $\varphi$.
\end{proof}

Comparing Proposition \ref{prop:S_S_rieman_proj_rel} and Proposition
\ref{prop:S_S_randers} we obtain the following result.

\begin{corollary}
  The geodesics of the natural spray and those of the Randers metric $\F$
  associated with the navigation data $(h,W)$ coincide if and only if they
  are projectively equivalent to the geodesics of the Riemann metric $h$.
\end{corollary}

\begin{remark}
  \label{rem:isotropic_S_curv}
  According to \cite{Cheng_Shen_2009}, a Randers metric has isotropic
  $S$-curvature if and only if
  \begin{math}
    R_{00}= \varphi h^2
  \end{math}
  with some function $\varphi \in C^\infty(M)$. It follows the following
\end{remark}

\begin{corollary}
  If the natural spray and the geodesic spray of the Randers metric
  associated to a navigation data coincide, then the Randers metric has
  isotropic $S$-curvature.
\end{corollary}

\begin{proof}
  As shown in the proof of Proposition \ref{prop:S_S_randers}, if the
  natural and the Randers sprays associated to a navigation data coincide,
  then \eqref{eq:R_2} holds. By Remark \ref{rem:isotropic_S_curv}, this
  implies that the Randers metric has isotropic $S$-curvature.
  
\end{proof}

\bigskip

\section{A note on metric correction processes}
\label{sec:6}

The purpose of a metric correction process is to adapt a connection to a
metric environment. The problem is closely related to the history of
Finsler geometry, the first challenge of which was the development of a
suitable concept of connection. Many attempts to solve the problem have
been made and are being made in certain special cases up to this day (see,
for example, the theory of generalized Berwald spaces). In the words of
M. Matsumoto: \emph{``There is a most suitable Finsler connection for every
  geometrical formulation''}. In case of the parallel transport
\eqref{eq:natural_parallel} the process is based on the vector field $W$ as
one of the initial data of a navigation problem. Especially, the
construction \eqref{eq:natural_parallel} can be applied to any metric
linear connection of the Riemannian metric $h$. Therefore the correction
changes metric linear connections given by the initial data of a navigation
problem to Finslerian non-linear metric connections, where the Finslerian
metric function is of Randers type derived from the navigation
problem. Results that are independent of the torsion remain valid. The
following example shows a different metric correction process in a more
general case. The detailed explanation of such a translate-and-normalize
process can be found in \cite{KozmaBaran} using the general context of
Finsler vector bundles.

\begin{example} Let $F$ be a Finsler metric function and consider a linear connection $\nabla^0$ on the base manifold. The metric correction process
\begin{equation}
    \label{eq:Finlser_parallel_gen}
    \P^c\colon T_pM\to T_qM, \quad \P^c(V_p):= \frac{F(V_p)}{F(\P^{0}(V_p))}\P^0(V_p),
  \end{equation}
  where $\P^0$ is the parallel translation along $c$ with respect to
  $\nabla^0$, gives a homogeneous (non-linear) parallel translation between
  the tangent spaces keeping the Finslerian metric function invariant. Let
  $V_t:=\P^c_t(V_p)$ be a parallel vector field and consider its derivative
  at $t=0$ to give the induced horizontal distribution in terms of the
  connection coefficients:
  \begin{displaymath}
    \frac{d V_t^k}{dt}=F(V_p) \left(-\frac{1}{F^2(\P^0_t (V_p))}\frac{d
        F(\P^0_t (V_p))}{dt}[\P_t^0(V_p)]^k
      +\frac{1}{F(\P^0_t (V_p))}\frac{d [\P^0_t (V_p)]^k}{dt}\right),
  \end{displaymath}
  where
  \begin{alignat*}{1}
    \frac{d F(\P^0_t (V_p))}{dt}
    &=\frac{\partial F}{\partial x^l} \circ \P^0_t (V_p)
      \left(x^l\circ \P^0_t   (V_p)\right)' +\frac{\partial F}
      {\partial y^l}\circ \P^0_t (V_p)\left(y^l\circ \P^0_t (V_p)\right)'
    \\
    & =\dot{c}^l(t)\frac{\partial F}{\partial x^l}\circ \P^0_t (V_p)
      -\dot{c}^{i}(t)V_t^j A_{ij}^l(c(t))\frac{\partial F}{\partial y^l}
      \circ \P^0_t  (V_p)=\dot{c}^i(t) \left(\delta_i^0F\right)
      \circ \P^0_t (V_p),
  \end{alignat*}
  where the functions $A_{ij}^k$ are the Christoffel symbols of
  $\nabla^0$ and 
  \begin{displaymath}
    \delta_i^0=\frac{\partial}{\partial x^i}-V_t^j A_{ij}^l(c(t))
    \frac{\partial}{\partial y^l}.
  \end{displaymath}
  Therefore
  \begin{displaymath}
    \frac{d V_t^k}{dt} \bigg |_{t=0}=-\dot{c}^i(0)
    \left(\frac{\delta_i^0 F}{F}(V_p)V^k(p)+V^j(p) A^k_{ij}(c(0))\right)
  \end{displaymath}
  and, consequently, the horizontal distributions are related as
  \begin{displaymath}
    \Gamma_i^k=\frac{\delta_i^0 F}{F} y^k+y^j A^k_{ij} \ \  \Rightarrow \ \
    \mathfrak{h}=\mathfrak{h}^0-\frac{d_{\mathfrak{h}^0}F}{F}\otimes
    \mathcal C.
  \end{displaymath}
  The associated sprays
  \begin{displaymath}
    S=S^0-\frac{S^0F}{F} \C 
  \end{displaymath}
  are projectively equivalent. 
\end{example}

\medskip

\textbf{Acknowledgements:} The authors would like to thank the reviewers
for their constructive comments, which helped to improve the article.


\medskip

\begin{thebibliography}{10}

\bibitem{Mezrag_Muzsnay_2024}
  M.~Asma and Z.~Muzsnay.
  \newblock The holonomy of spherically symmetric projective {F}insler
  metrics of constant curvature.
  \newblock {\em J. Geom. Anal.}, 34(8):Paper No. 257, pp. 15, 2024.

\bibitem{Bao_Robles_Shen_2004}
  D.~Bao, C.~Robles, and Z.~Shen.
  \newblock Zermelo navigation on {R}iemannian manifolds.
  \newblock {\em J. Differential Geom.}, 66(3):377–435, 2004.

\bibitem{Cheng_Shen_2009} X.~Cheng, Z.~Shen.
  \newblock Randers metrics of scalar flag curvature.
  \newblock {\em J. Aust. Math. Soc.}, 87:359–370, 2009.


\bibitem{de_rham_1952}
  G.~de~Rham.
  \newblock Sur la reductibilit\'e{} d'un espace de {R}iemann.
  \newblock {\em Comment. Math. Helv.}, 26:328--344, 1952.

\bibitem{Ehresmann_1995}
  C.~Ehresmann.
  \newblock Les connexions infinit\'esimales dans un espace fibr\'e{}
  diff\'erentiable.
  \newblock In {\em S\'eminaire {B}ourbaki, {V}ol.\ 1}, pages Exp. No. 24,
  153--168. Soc. Math. France, Paris, 1995.

\bibitem{Elgendi_Muzsnay_2017} S.~G. {Elgendi} and Z.~{Muzsnay}.
  \newblock Freedom of {$h(2)$}-variationality and metrizability of sprays.
  \newblock {\em Differential Geom. Appl.}, 54(part A):194–207, 2017.

\bibitem{Eschenburg_Heintze_1998}
  J.-H. Eschenburg and E.~Heintze.
  \newblock Unique decomposition of {R}iemannian manifolds.
  \newblock {\em Proc. Amer. Math. Soc.}, 126(10):3075--3078, 1998.

\bibitem{Grifone_1972} J.~Grifone.
  \newblock Structure presque-tangente et connexions. {I}.
  \newblock {\em Ann. Inst. Fourier (Grenoble)}, 22(1):287–334, 1972.

\bibitem{Hubicska_Matveev_Muzsnay_2021} B.~Hubicska, V.~S. Matveev, and
  Z.~Muzsnay.
  \newblock Almost all Finsler metrics have infinite dimensional holonomy
  group.
  \newblock {\em The Journal of Geometric Analysis}, 31(6):6067--6079, June
  2021.

\bibitem{Hubicska_Muzsnay_2020a} B.~{Hubicska} and Z.~{Muzsnay}.
  \newblock The holonomy groups of projectively flat Randers two-manifolds
  of constant curvature.
  \newblock {\em Differential Geometry and its Applications}, 73:101677,
  2020.

\bibitem{Kobayashi_Nomizu_1996} S.~Kobayashi and K.~Nomizu.
  \newblock {\em Foundations of differential geometry. {V}ol. {I}}.
  \newblock Wiley Classics Library. John Wiley \& Sons, Inc., New York, 1996.
  \newblock Reprint of the 1963 original, A Wiley-Interscience Publication.

\bibitem{KozmaBaran}
  L.~Kozma and S.~Baran.
  \newblock On metrical homogeneous connections of a Finsler point space.
  \newblock {\em Publicationes Mathematicae Debrecen}, 49(1-2):59--68,
  1996.

\bibitem{Muzsnay_Nagy_max_2015}
  Z.~Muzsnay and P.~T. Nagy.
  \newblock Finsler 2-manifolds with maximal holonomy group of infinite
  dimension.
  \newblock {\em Differential Geom. Appl.}, 39:1–9, 2015.

\bibitem{Okada_1983} T. Okada.
  \newblock{On models of projectively flat Finsler spaces with constant
    negative curvature}
  \newblock{\em Tensor (NS)}, 40: 117--123, 1983.

\bibitem{robles_2007}
  C.~Robles.
  \newblock Geodesics in Randers spaces of constant curvatur.
  \newblock {\em Transactions of the American Mathematical Society},
  359(4):1633 -- 1651, 2007.

\bibitem{Spivak_1979}
  M.~Spivak.
  \newblock {\em A comprehensive introduction to differential
    geometry. {V}ol.  {I}}.
  \newblock Publish or Perish, Inc., Wilmington, Del., second edition, 1979.

\bibitem{Vincze2015} C.~Vincze.
  \newblock On Randers manifolds with semi-symmetric compatible linear
  connections.
  \newblock {\em Indagationes Mathematicae}, 26(2):363--379, 2015.

\end{thebibliography}
\end{document}